\input ./style/arxiv-vmsta.cfg
\documentclass[numbers,compress,v1.0.1]{vmsta}
\usepackage{vtexbibtags}

\volume{5}
\issue{2}
\pubyear{2018}
\firstpage{129}
\lastpage{143}
\aid{VMSTA99}
\doi{10.15559/18-VMSTA99}
\articletype{research-article}

\usepackage{authorquery}
\usepackage{mathbh}

\startlocaldefs

\allowdisplaybreaks

\theoremstyle{definition}

\def\ind{\mathbh{1}}
\newtheorem{ex}{Example}
\newtheorem{remark}{Remark}

\theoremstyle{plain}
\newtheorem{thm}{Theorem}

\newtheorem{cor}{Corollary}

\endlocaldefs

\begin{aqf}
\querytext{Q1}{Maybe "in contrast to the conditions ..." could be
suitable here.}
\querytext{Q2}{We read here that "something" is impossible (and the
last example shows it), but before Example 3 we read that it "shows the
possibility". It would be worth to display more clearly to a reader
what is possible, and what is impossible.}
\end{aqf}
\begin{document}

\begin{frontmatter}
\pretitle{Research Article}

\title{Exponential bounds for the tail probability of the supremum of
an inhomogeneous random walk}

\author{\inits{D.}\fnms{Dominyka}~\snm{Kievinait\.{e}}\ead[label=e1]{d.kievinaite@gmail.com}}
\author{\inits{J.}\fnms{Jonas}~\snm{\v Siaulys}\thanksref{cor1}\thanksref{f1}\ead[label=e2]{jonas.siaulys@mif.vu.lt}}
\thankstext[type=corresp,id=cor1]{Corresponding author.}
\thankstext[id=f1]{The second author was supported by grant No S-MIP-17-72 from the Research Council of Lithuania.}
\address{Faculty of Mathematics and Informatics, \institution{Vilnius University}, Naugarduko 24, Vilnius~LT-03225, \cny{Lithuania}}

\markboth{D. Kievinait\.{e}, J. \v Siaulys}{Exponential bounds for the tail probability of the supremum of an inhomogeneous random walk}

\begin{abstract}
Let $\{\xi_1,\xi_2,\ldots\}$ be a sequence of
independent but not necessarily identically distributed random variables.
In this paper, the sufficient conditions are found under which the tail
probability $\mathbb{P} (\sup_{n\geqslant0}\sum_{i=1}^n\xi_i>x
)$ can be bounded above by $\varrho_1\exp\{-\varrho_2 x\}$ with some
positive constants $\varrho_1$ and $\varrho_2$. A way to calculate
these two constants is presented. The application of the derived bound
is discussed and a Lundberg-type inequality is obtained for the
ultimate ruin probability in the inhomogeneous renewal risk model
satisfying the net profit condition on average.
\end{abstract}

\begin{keywords}
\kwd{Exponential bound}
\kwd{supremum of sums}
\kwd{tail probability}
\kwd{risk model}
\kwd{inhomogeneity}
\kwd{ruin probability}
\kwd{Lundberg's inequality}
\end{keywords}

\begin{keywords}[MSC2010]%
\kwd{62E20}
\kwd{60E15}
\kwd{60G50}
\kwd{91B30}
\end{keywords}

\received{\sday{13} \smonth{10} \syear{2017}}
\revised{\sday{10} \smonth{2} \syear{2018}}
\accepted{\sday{13} \smonth{2} \syear{2018}}
\publishedonline{\sday{15} \smonth{3} \syear{2018}}
\end{frontmatter}

\section{Introduction}\label{ii}

Let $\{\xi_1,\xi_2,\ldots\}$ be a sequence of independent real-valued
random variables (r.v.'s), and let
\[
M_\infty=\sup\limits
_{n\geqslant0} \Biggl\{\sum\limits
_{k=1}^n \xi_k \Biggr\}.
\]
Here and subsequently, all empty sums are assumed to be zero.\vadjust{\eject}

Sgibnev \cite{sg-1997} generalized results by Kiefer and Wolfowitz
\cite{k+w-1956} obtaining the upper bound for the submultiplicative
moment $\mathbb{E} (\varphi(M_\infty) )$ in the case of
independent and identically distributed (i.i.d.) r.v.'s. In Theorem 2
of that paper, the following assertion is proved.
%
\begin{thm}\label{tt1}Let $\{\xi_1,\xi_2,\ldots\}$ be a
sequence of i.i.d. r.v.'s with distribution function (d.f.) $F$, and
let $\varphi$ be a nondecreasing submultiplicative function defined on
$[0,\infty)$. Then $\mathbb{E} (\varphi(M_\infty) )<\infty$ under
the following conditions:
\begin{eqnarray*}
&\bullet& \mathbb{E}\xi_1<0,
\\
&\bullet& \int_0^\infty\varphi(x)\overline{F}(x)\,{\rm d}x<\infty,
\\
&\bullet& \mathbb{E} \bigl({\rm e}^{r\xi_1} \bigr)<1 \quad \text{if}\ r:=\lim\limits_{x\rightarrow\infty}\frac{\log\varphi(x)}{x}>0 .
\end{eqnarray*}
\end{thm}

Recall that a function $\varphi$ defined on the interval $[0,\infty)$
is said to be submultiplicative if
\[
\varphi(0)=1\quad \text{and}\quad \varphi(x+y)\leqslant\varphi(x)\varphi(y)\quad \text{for\ all}\ x,y \in[0,\infty).
\]
Theorem~\ref{tt1} was applied several times to find the asymptotic
behavior of the ruin probabilities in the homogeneous renewal risk models.

\textit{We say that the insurer's surplus process $R(t)$ varies
according to the homogeneous renewal risk model if
%
\begin{equation}
\label{a1} R(t)=u+pt-\sum_{i=1}^{\varTheta(t)}Z_i,
\quad t\geqslant0,
\end{equation}
where:
\begin{itemize}\setlength\itemsep{0pt}
\item[$\bullet$] $u\geqslant0$ denotes the initial insurer's surplus;
\item[$\bullet$] $ p>0$ denotes a constant premium rate;
\item[$\bullet$] the claim sizes $\{Z_1,Z_2,\ldots\}$ form a sequence of i.i.d.
nonnegative r.v.'s;
\item[$\bullet$] $\varTheta(t)=\sum_{n=1}^\infty\ind_{\{\theta
_1+\theta_2+\cdots+\theta_n\,\leqslant\, t\}}$ is the renewal counting
process generated by the inter-oc\-cur\-ren\-ce times $\{\theta_1,\theta
_2,\ldots\}$, which form another sequence of i.i.d. nonnegative and
nondegenerate at $ 0$ r.v.'s;
\item[$\bullet$] the sequences $\{Z_1,Z_2, \ldots\}$ and $\{\theta
_1,\theta_2,\ldots\}$ are mutually independent.
\end{itemize}}

The ultimate ruin probability
%
\begin{equation}
\label{pp} \psi(u)=\mathbb{P} \Bigl(\inf\limits
_{t\geqslant0}R(t)<0 \Bigr)=\mathbb {P}
\Biggl(\sup\limits
_{n\geqslant1}\sum\limits
_{k=1}^n(Z_k-p
\theta _k)>u \Biggr)
\end{equation}
and the probability of ruin within time $T$
\[
\psi(u,T)=\mathbb{P} \Bigl(\inf\limits
_{0\leqslant t\leqslant
T}R(t)<0 \Bigr)=\mathbb{P} \Biggl(\sup
\limits
_{1\leqslant n\leqslant\varTheta
(T)}\sum\limits
_{k=1}^n(Z_k-p
\theta_k)>u \Biggr)
\]
are the main characteristics of the renewal risk model.\vadjust{\eject}

The asymptotic behavior of $\psi(u,T)$ was considered by Tang \cite
{tang-2004} when random claims in the homogeneous model have d.f. with
consistently varying tail. The author of the paper uses the assertion
of Theorem~\ref{tt1} for function $\varphi(x)=(1+x)^q$ with $q>0$ to
get the main term of the asymptotics for the probability $\psi(u,T)$.
Leipus and \v Siaulys considered the asymptotic behavior of $\psi
(u,T)$ in
\cite{ls-2007,ls-2009} but for subexponentially distributed r.v.'s $\{
Z_1,Z_2,\ldots\}$. In their proofs, the assertion of Theorem~\ref{tt1}
was used for function $\varphi(x)=\exp(\rho x)$ with some $\rho>0$ (see
Lemma 3.3 in \cite{ls-2007} and Lemma 2.1 in \cite{ls-2009}). In the
case of exponential function, Theorem~\ref{tt1} implies the following assertion.

\begin{cor}\label{cca1}
Let $\{\xi_1,\xi_2,\ldots\}$ be a
sequence of i.i.d. r.v.'s. If $\mathbb{E}\xi_1<0$ and $\mathbb{E}\,{\rm
e}^{h \xi_1}<\infty$ for some positive $h$, then there exists a
positive constant $\varrho$ such that
\[
{\rm e}^{\varrho x}\mathbb{P} (M_\infty>x )\mathop{\rightarrow }
\limits_{x\rightarrow\infty}0.
\]
\end{cor}

The Sgibnev's proof of Theorem~\ref{tt1} is substantially related to
the techniques of Banach algebras, while Corollary~\ref{cca1} can be
derived using only the probabilistic approach. Wang et al. (see Lemma
4.4 in \cite{wcwm-2012}) demonstrated such a probabilistic way to
obtain the assertion of Corollary~\ref{cca1} supposing, in addition,
that r.v.'s $\{\xi_1,\xi_2,\ldots\}$ follow some dependence structure.
Corollary~\ref{cca1} can be applied not only as auxiliary assertion in
the consideration of the asymptotic behavior of $\psi(u,T)$. The
assertion of Corollary~\ref{cca1} is closely related to the following
statement on the upper bound for $\psi(u)$ in the homogeneous renewal
risk model.

\begin{thm}\label{tt2}
Let the
claim sizes $\{Z_1,Z_2,\ldots\}$ and the inter-oc\-cur\-ren\-ce times\break
$\{\theta_1,\theta_2,\ldots\}$ form a homogeneous renewal risk model.
Let, in addition, the net profit condition $\mathbb{E}Z_1-p\,\mathbb
{E}\theta_1<0$ hold and $\mathbb{E}\,{\rm e}^{hZ_1}<\infty$ for some
positive $h$. Then, there exists a positive $H$ such that
%
\begin{equation}
\label{ll1} \psi(u)\leqslant{\rm e}^{-H\,u},\quad u\geqslant0.
\end{equation}
If $\ \mathbb{E}\,{\rm e}^{R\,(Z_1-p\,\theta_1)}=1$ for some positive
$R$, then we can take $H=R$ in \eqref{ll1}.
\end{thm}

The above assertion is the well-known Lundberg inequality. There exist
a lot of different proofs of this inequality. For example, some of the
proofs can be found in \cite{AA-2010}, \cite{EKM-1997}, \cite{ev-1982},
\cite{g-1973}, \cite{M-2009}. The existing proofs of Lundberg's
inequality are essentially based on the renewal idea. However, the
classical methods used for consideration of the ruin probability in the
homogeneous renewal risk model are not applicable in the case of the
inhomogeneous model because at any time moment distribution of the
future is completely new.

Another way to derive the Lundberg inequality is related to Theorem~\ref{tt1}. Namely, the first part of Theorem~\ref{tt2} follows from
Theorem~\ref{tt1} and the additional inequality $\psi(0)<1$. We use
this approach to get the inequality similar to the Lundberg inequality
but for the inhomogeneous renewal risk model with not necessarily
identically distributed claim sizes $\{Z_1,Z_2,\ldots\}$ and the
inter-occurrence times $\{\theta_1,\theta_2,\ldots\}$.

In this paper, we consider a sequence of independent but not
necessarily identically distributed r.v.'s $\{\xi_1,\xi_2,\ldots\}$.
We obtain an assertion similar to that in Corollary~\ref{cca1}. We
present an algorithm to get the numerical values of the two positive
constants in the exponential bound for $\mathbb{P}(M_\infty>x)$ in the
case of not necessarily identically distributed r.v.'s $\{\xi_1,\xi
_2,\ldots\}$. We apply the obtained estimate to derive two
Lundberg-type inequalities similar to that in Theorem~\ref{tt2} but for
the inhomogeneous renewal risk model with possibly nonidentically
distributed random claim amounts $Z_1,Z_2,\ldots$. Corollaries~\ref
{ccc1} and~\ref{ccc2} below show that the exponential bound for the
ruin probability in the inhomogeneous renewal risk model holds if the
model satisfies the \textit{net profit condition on average}. This
means that the quantity
\[
\frac{1}{n}\sum_{k=1}^n\mathbb{E}
(Z_k-p\theta_k )
\]
is negative for all sufficiently large $n$.

The results of the present paper are complementary to those obtained
by Albrecher et al. \cite{aiz-2016}, Ambagaspitiya \cite{am-2009},
Bernackait\.{e} and \v{S}iaulys \cite{bs-2015,bs-2017}, Casta\~{n}er
et al. \cite{ccglm-2013}, Cojocaru \cite{co-2017}, Constantinescu et
al. \cite{ckm-2013}, Czarna and Palmowski \cite{cp-2011}, Damarackas
and \v{S}iaulys \cite{ds-2014}, Danilenko et al. \cite{dms-2017},
Grigutis et al. \cite{gks-2015,gks-2015+}, Jordanova and Stehl\'{i}k
\cite{js-2016}, Mishura et al. \cite{mrs-2014}, R\u{a}ducan et al. \cite
{rvz-2015,rvz-2015b,rvz-2016}, Ragulina \cite{r-2017}, Zhang et al.
\cite{zfly-2017}, Zhang et al. \cite{zcy-2017} and other authors who
dealt with different inhomogeneous risk models.

The rest of the paper is organized as follows. In Section~\ref{main},
we present our main result together with its proof. In Section~\ref{ren}, we recall the concept of the inhomogeneous renewal risk model
and we present two corollaries from the main theorem, which yield
exponential bounds for the ruin probability in this model. Finally, in
Section~\ref{ex}, we present some examples which show the applicability
of the theorem and corollaries.

\section{Main result}\label{main}

In this section, we formulate and prove our main result. The assertion
below is a generalization of Lemma 1 by Andrulyt\.{e} et al. \cite
{abks-2015}. In that lemma, the exponential bound for $ \mathbb
{P}(M_\infty>x)$ was established under more restrictive conditions. In
addition, the assertion below provides an algorithm to calculate two
positive constants establishing this exponential bound. For this
reason, the conditions of the main theorem are formulated in an
explicit form\querymark{Q1} in contrast to  the conditions of Lemma
1 in \cite{abks-2015}. It should be noted that the presented proof of
the main result has some similarities with the classical approach by
Chernoff \cite{ch-1952} and Hoeffding \cite{ho-1963}.

\begin{thm}\label{th1} Let $\{\xi_1,\xi_2,\ldots\}$ be
independent r.v.'s such that:
\begin{eqnarray*}
&(\textnormal{i})&\frac{1}{n}\sum\limits
_{i=1}^{n}
\mathbb{E}\, \xi _{i}\leqslant-a \quad \text{if}\ n\geqslant b,
\\
&(\textnormal{ii})&\sup_{n\geqslant b}\frac{1}{n} \sum
\limits
_{i=1}^{n} \mathbb{E}\big(|\xi_{i}|
\ind_{\lbrace\xi_{i}\leqslant-c\rbrace} \big)\leqslant\varepsilon,
\\
&(\textnormal{iii})&\sup_{n\geqslant b}\frac{1}{n}\sum
\limits
_{i=1}^{n} \bigl(\mathbb{P}(\xi_{i}\leqslant0)+
\mathbb{E} \bigl({\rm e}^{h\xi_{i}}\ind_{\{\xi_i>0\}} \bigr) \bigr)\leqslant
d_1,
\\
&(\textnormal{iv})&\max_{1\leqslant n\leqslant b-1}\,\frac{1}{n}\sum
\limits
_{i=1}^{n} \bigl(\mathbb{P}(\xi_{i} \leqslant0)+
\mathbb{E} \bigl({\rm e}^{h\xi_{i}}\ind_{\{\xi_i>0\}} \bigr) \bigr)\leqslant
d_2,\
\end{eqnarray*}
for some $a>0$, $b\in\mathbb{N}$, $c>0$, $\varepsilon\geqslant0$,
$h>0$, $d_1\geqslant1$ and $d_2\geqslant1$.

If
%
\begin{equation}
\label{0} -\varDelta:=\varepsilon+\delta h d_1\max \biggl\{
\frac{c^2}{2},\frac
{2}{h^2} \biggr\}-a<\xch{0,}{0}
\end{equation}
with some $\delta\in(0, 1/2]$, then
\[
\mathbb{P} \Biggl(\displaystyle\sup_{n\geqslant0}\sum
_{i=1}^n \xi _{i}>x \Biggr)\leqslant\min
\bigl\{1,\, c_{1}{\rm e}^{-\delta{h} x} \bigr\} ,\quad x\geqslant0,
\]
where
\[
c_1= \xch{\biggl(S(b,d_2)+\frac{\exp\{-\delta{h}\varDelta b\}}{1-\exp\{-\delta
{h}\varDelta\}} \biggr),}{\biggl(S(b,d_2)+\frac{\exp\{-\delta{h}\varDelta b\}}{1-\exp\{-\delta
{h}\varDelta\}} \biggr)}
\]
with
\[
S(b,d_2)= %
\begin{cases}d_2\frac{d_2^{b-1}-1}{d_2-1}\quad \text{if}\ d_2>1,\\
b-1\quad \text{if}\ d_2=1.
\end{cases} %
\]
\end{thm}

\begin{proof} We observe that, for all $x\geqslant0$,
\begin{align*}
\mathcal{P}(x)&:=\mathbb{P} \Biggl(\displaystyle\sup_{n\geqslant0}
\sum_{i=1}^n \xi_{i}>x \Biggr)=
\mathbb{P} \Biggl(\bigcup_{n=1}^\infty \Biggl
\lbrace\sum_{i=1}^n \xi_{i}>x
\Biggr\rbrace \Biggr)
\\
&\leqslant \sum_{n=1}^\infty\mathbb{P}
\Biggl(\sum_{i=1}^n \xi_{i}>x
\Biggr).
\end{align*}

Since r.v.'s $\{\xi_1,\xi_2,\ldots\}$ are independent, by the exponential Chebyshev inequality, we get
%
\begin{align}
\label{1}
\mathcal{P}(x)&\leqslant{\rm e}^{-yx}\sum\limits_{n=1}^\infty\prod_{i=1}^{n}\mathbb{E}\,{\rm e}^{y\xi_{i}}\nonumber\\
&={\rm e}^{-yx}\sum\limits_{n=1}^{b-1}\prod_{i=1}^{n}\mathbb{E}\,{\rm e}^{y\xi_{i}}+{\rm e}^{-yx}\sum\limits_{n=b}^\infty\prod_{i=1}^{n}\mathbb{E}\,{\rm e}^{y\xi_{i}}\nonumber\\
&:=\mathcal{P}_1(x)+\mathcal{P}_2\xch{(x),}{(x)}
\end{align}
for all $x\geqslant0$ and $0<y\leqslant h$.

For all $i\in\mathbb{N}$, we have
\begin{align*}
\mathbb{E} {\rm e}^{y\xi_{i}}&= 1+y\mathbb{E}\xi_{i}+\mathbb{E}
\bigl( \bigl({\rm e}^{y\xi_{i}}-1 \bigr)\ind_{\lbrace\xi_{i}\leqslant-c\rbrace} \bigr)-y\mathbb{E}
\bigl( \xi_{i}\ind_{\lbrace\xi_{i}\leqslant-c\rbrace}\bigr)
\\
&\quad +\mathbb{E} \bigl( \bigl({\rm e}^{y\xi_{i}}-1-y\xi_{i} \bigr)
\ind_{\lbrace-c<\xi
_{i}\leqslant0\rbrace} \bigr)
\\
&\quad +\mathbb{E} \bigl( \bigl({\rm e}^{y\xi_{i}}-1-y\xi_{i} \bigr)
\ind_{\lbrace\xi
_{i}>0\rbrace} \bigr).
\end{align*}

It is obvious that
\begin{align*}
{\rm e}^{v}-1&\leqslant\,0\quad \text{if}\ v \leqslant0,\\
{\rm e}^{v}-v-1&\leqslant\frac{v^2}{2}\quad \text{if}\ v \leqslant 0, \\
{\rm e}^{v}-v-1&\leqslant\frac{v^2}{2}\,{\rm e}^{v}\quad \text{if}\ v\geqslant0,
\end{align*}
and $v^2\leqslant{\rm e}^v$ for nonnegative $v$. Using these inequalities we get
%
\begin{align}
\label{2} \mathbb{E} {\rm e}^{y\xi_{i}} &\leqslant 1+y\mathbb{E}
\xi_i+y\,\mathbb {E} \bigl(|\xi_{i}|\ind_{\lbrace\xi_{i}\leqslant-c\rbrace}\bigr) +
\frac{y^{2}}{2}\,\mathbb{E} \bigl(\xi_{i}^2
\ind_{\lbrace-c<\xi
_{i}\leqslant0\rbrace} \bigr)
\nonumber
\\
&\quad +\ \frac{y^{2}}{2}\,\mathbb{E} \bigl(\xi_{i}^{2}{
\rm e}^{y\xi_{i}}\ind_{\lbrace\xi_{i}>0\rbrace} \bigr)
\nonumber
\\
&\leqslant 1+y\mathbb{E}\xi_i+y\,\mathbb{E}\bigl(|\xi_{i}|
\ind_{\lbrace
\xi_{i}\leqslant-c\rbrace} \bigr) + \frac{y^2c^2}{2}\,\mathbb{P}(\xi_i
\leqslant0)
\nonumber
\\
&\quad + \frac{2y^2}{h^2}\,\mathbb{E} \xch{\bigl({\rm e}^{h\xi_{i}}\ind_{\lbrace\xi_{i}>0\rbrace} \bigr),}{\bigl({\rm e}^{h\xi_{i}}\ind_{\lbrace\xi_{i}>0\rbrace} \bigr)}
\end{align}
if $0<y\leqslant h/2$, because
\[
\mathbb{E} \biggl( \biggl(\frac{h \xi_i}{2} \biggr)^2{\rm
e}^{y\xi_{i}}\ind _{\lbrace\xi_{i}>0\rbrace} \biggr)\leqslant\mathbb{E} \biggl( \biggl(
\frac{h \xi
_i}{2} \biggr)^2{\rm e}^{h\xi_{i}/2}\ind_{\lbrace\xi_{i}>0\rbrace}
\biggr)\leqslant\mathbb{E} \xch{\bigl({\rm e}^{h\xi_{i}}\ind_{\lbrace\xi_{i}>0\rbrace
} \bigr),}{\bigl({\rm e}^{h\xi_{i}}\ind_{\lbrace\xi_{i}>0\rbrace
} \bigr)}
\]
in this case.

If $n\geqslant b$, then conditions (ii), (iii) and relation \eqref{2}
together with the inequality $1+u\leqslant{\rm e}^u$, $u\in\mathbb
{R}$, imply that
\begin{align*}
\prod\limits
_{i=1}^{n}\mathbb{E} {\rm e}^{y\xi_{i}}
&\leqslant\prod\limits
_{i=1}^{n} \biggl(1+y\mathbb{E}
\xi_i+y\,\mathbb {E} \bigl(|\xi_{i}|\ind_{\lbrace\xi_{i}\leqslant-c\rbrace}\bigr)
\\
&\quad + \frac{y^2c^2}{2}\,\mathbb{P}(\xi_i\leqslant0)+
\frac{2y^2}{h^2}\, \mathbb{E} \bigl({\rm e}^{h\xi_{i}}\ind_{\lbrace\xi_{i}>0\rbrace}
\bigr) \biggr)
\\
&\leqslant\exp \Biggl\{y\sum\limits
_{i=1}^{n}
\mathbb{E}\xi_i+ny\,\sup\limits
_{n\geqslant b}\frac{1}{n}\sum
\limits
_{i=1}^{n}\mathbb{E} \bigl(|\xi_{i}|
\ind_{\lbrace\xi_{i}\leqslant-c\rbrace} \bigr)
\\
&\quad + ny^2\max \biggl\{\frac{c^2}{2},\frac{2}{h^2} \biggr\}
\sup\limits
_{n\geqslant b}\frac{1}{n}\sum\limits
_{i=1}^{n}
\bigl(\mathbb{P}(\xi _i\leqslant0)+\mathbb{E} \bigl({\rm
e}^{h\xi_{i}}\ind_{\lbrace\xi
_{i}>0\rbrace} \bigr) \bigr) \Biggr\}
\\
&\leqslant \exp \Biggl\{y\sum\limits
_{i=1}^{n}
\mathbb{E}\xi _i+ny\varepsilon+ny^2d_1\max
\biggl\{\frac{c^2}{2},\frac{2}{h^2} \biggr\} \Biggr\}.
\end{align*}

Hence, by condition (i)
%
\begin{equation}
\label{3} \mathcal{P}_2(x)\leqslant{\rm e}^{-yx}\sum
\limits
_{n=b}^{\infty}\exp \xch{\biggl\{ny \biggl(-a+
\varepsilon+yd_1\max \biggl\{\frac{c^2}{2},\frac
{2}{h^2}
\biggr\} \biggr) \biggr\},}{\biggl\{ny \biggl(-a+
\varepsilon+yd_1\max \biggl\{\frac{c^2}{2},\frac
{2}{h^2}
\biggr\} \biggr) \biggr\}}
\end{equation}
for all $x\geqslant0$ and $0<y\leqslant h/2$.

If $n\leqslant b-1$ and $0<y\leqslant h$ then, due to the condition
(iv), we have
\[
\frac{1}{n}\sum\limits
_{i=1}^{n}\mathbb{E} {\rm
e}^{y\xi_{i}}=\frac
{1}{n}\sum\limits
_{i=1}^{n}
\bigl(\mathbb{E} \bigl({\rm e}^{y\xi_{i}}\ind _{\{\xi_i\leqslant0\}} \bigr)+
\mathbb{E} \bigl({\rm e}^{y\xi_{i}}\ind_{\{
\xi_i> 0\}} \bigr) \bigr)\leqslant
d_2.
\]
Therefore,
%
\begin{equation}
\label{4} \mathcal{P}_1(x)\leqslant{\rm e}^{-yx}\sum
\limits
_{n=1}^{b-1} \xch{d_2^{\,n},}{d_2^{\,n}}
\end{equation}
because
\[
\prod\limits
_{i=1}^{n}\mathbb{E} {\rm e}^{y\xi_{i}}
\leqslant \xch{\Biggl(\frac
{1}{n}\sum\limits
_{i=1}^{n}
\mathbb{E} {\rm e}^{y\xi_{i}} \Biggr)^n,}{\Biggl(\frac
{1}{n}\sum\limits
_{i=1}^{n}
\mathbb{E} {\rm e}^{y\xi_{i}} \Biggr)^n}
\]
by the inequality of arithmetic and geometric means.

Equality \eqref{1} and inequalities \eqref{3}, \eqref{4} imply that
%
\begin{equation}
\label{5} \mathcal{P}(x)\leqslant{\rm e}^{-yx}
\xch{\Biggl(S(b,d_2)+\sum\limits
_{n=b}^{\infty}
\biggl(\exp \biggl\{y \biggl(-a+\varepsilon+yd_1\max \biggl\{
\frac{c^2}{2},\frac{2}{h^2} \biggr\} \biggr) \biggr\}
\biggr)^n \Biggr),}{\Biggl(S(b,d_2)+\sum\limits
_{n=b}^{\infty}
\biggl(\exp \biggl\{y \biggl(-a+\varepsilon+yd_1\max \biggl\{
\frac{c^2}{2},\frac{2}{h^2} \biggr\} \biggr) \biggr\}
\biggr)^n \Biggr)}
\end{equation}
for all $x\geqslant0$ and $0<y\leqslant h/2$.

Let now $y=\delta{h}$ with some $\delta\in(0,1/2]$ satisfying
condition \eqref{0}. For such $y$, we obtain from \eqref{5} that
\[
\mathcal{P}(x)\leqslant{\rm e}^{-\delta{h} x} \biggl(S(b,d_2)+
\frac{\exp\{
-\delta{h}\varDelta b\}}{1-\exp\{-\delta{h}\varDelta\}} \biggr).
\]

This is the desired inequality. The theorem is proved.
\end{proof}

\section{Lundberg-type inequalities}\label{ren}

In this section, we present two corollaries from Theorem~\ref{th1},
which yield the Lundberg-type inequalities for the inhomogeneous
renewal risk model.

\smallskip

\textit{We say that the insurer's surplus process $R(t)$ varies
according to the inhomogeneous renewal risk model if equality \eqref
{a1} holds for all $t\geqslant0$ with the initial insurer's surplus
$u\geqslant0$, a constant premium rate $p>0$, a sequence of
independent nonnegative and not necessarily identically distributed
claim amounts $\{Z_1,Z_2,\ldots\}$ and with the renewal counting
process $\varTheta(t)$ generated by the inter-oc\-cur\-ren\-ce times $\{
\theta_1,\theta_2,\ldots\}$, which form a sequence of independent
nonnegative nondegenerate at zero and possibly not identically
distributed r.v.'s. In addition, sequences $\{Z_1,Z_2,\ldots\}$ and $\{
\theta_1,\theta_2,\ldots\}$ are supposed to be independent.}

\smallskip

It is obvious that definitions and expressions of the ruin
probabilities $\psi(u)$ and $\psi(u,T)$ for the inhomogeneous renewal
risk model remain the same as those given in Section~\ref{ii}.

The main requirement to get the Lundberg-type bounds for $\psi(u)$ is
the \textit{net profit condition}. In both assertions below, it is
supposed that this condition holds on average. Our first corollary
follows immediately from Theorem~\ref{th1} and representation \eqref{pp}.

\begin{cor}\label{ccc1}
Let us consider the inhomogeneous renewal risk model such that
r.v.'s $\xi_k=Z_k-p\theta_k$, $k\in\mathbb{N}$, satisfy conditions
{\rm(i)}--{\rm(iv)} of Theorem~\ref{th1}. Then the ruin probability
in the model satisfies the following inequality
\[
\psi(u)\leqslant\min \bigl\{1,\,c_{1}{\rm e}^{-\delta{h} u} \bigr\},\
u\geqslant0,
\]
where constants $h>0$, $\delta\in(0,1/2]$ and $c_1>0$ are the same as
in Theorem~\ref{th1} for the sequence $\{\xi_1=Z_1-p\theta_1, \xi
_2=Z_2-p\theta_2,\ldots\}$.
\end{cor}

Our second corollary is more convenient to use because the
requirements are formulated separately for r.v.'s $\{Z_1,Z_2,\ldots\}$
and $\{\theta_1,\theta_2,\ldots\}$ in it. We present the corollary
below together with a short proof.

\begin{cor}\label{ccc2}
Let the inhomogeneous renewal risk model with a sequence of
random claim amounts $\{Z_1,Z_2,\ldots\}$, a sequence of random
inter-occurrence times\break $\{\theta_1,\theta_2,\ldots\}$ and premium rate
$p$ satisfy the following additional requirements
\begin{eqnarray*}
&{\rm(i)}& \frac{1}{n}\sum\limits_{i=1}^{n}( \hspace{0.5mm}\mathbb{E} Z_{i}-p\,\mathbb {E}\theta_i)\leqslant-\alpha\quad \text{if}\ n\geqslant\beta,\\
&{\rm(ii)}&\sup_{n\geqslant\beta}\frac{1}{n} \sum\limits_{i=1}^{n}\mathbb{E} \bigl(\theta_i\ind_{ \{\theta_i\geqslant\frac{ \varkappa}{p} \} } \bigr)\leqslant\epsilon,\\
&{\rm(iii)}& \sup_{n\geqslant\beta}\frac{1}{n}\sum\limits_{i=1}^{n}\mathbb{E} {\rm e}^{\gamma Z_i}\leqslant\nu_1,\\
&{\rm(iv)}& \max_{1\leqslant n\leqslant\beta-1}\,\frac{1}{n}\sum\limits_{i=1}^{n} \mathbb{E} {\rm e}^{\gamma Z_i}\leqslant\nu_2,
\end{eqnarray*}
for some $\alpha>0$, $\beta\in\mathbb{N}$, $\varkappa>0$, $\epsilon
\geqslant0$, $\gamma>0$, $\nu_1\geqslant1$ and $\nu_2\geqslant1$.

If
\[
-\hat{\varDelta}:=p\epsilon+\delta\gamma(1+\nu_1)\max \biggl\{
\frac{\varkappa
^2}{2},\frac{2}{\gamma^2} \biggr\}-\alpha<\xch{0,}{0}
\]
for some $\delta\in(0, 1/2]$, then
\[
\psi(u)\leqslant\min \bigl\{1,\, c_2{\rm e}^{-\delta\gamma\, u} \bigr\},\quad u\geqslant0,
\]
with the positive constant
\[
c_2= \biggl(\,\frac{1+\nu_2}{\nu_2} \bigl((1+\nu_2)^{b-1}-1
\bigr)+\frac{\exp
\{-\delta\gamma\hat{\varDelta} \beta\}}{1-\exp\{-\delta\gamma\hat{\varDelta}\}
} \biggr).
\]
\end{cor}

\begin{proof}[Proof of Corollary~\ref{ccc2}] Let $\xi_i=Z_i-p\,\theta
_i$ for all $i\in\mathbb{N}$. Then obviously
%
\begin{equation}
\label{cc1} \frac{1}{n}\sum\limits
_{i=1}^{n}
\mathbb{E}\xi_i\leqslant-\alpha\quad \text{if}\ n\geqslant\xch{\beta,}{\beta}
\end{equation}
by condition (i) of the corollary.

Further, by conditions (iii) and (iv), we have
%
\begingroup
\abovedisplayskip=7.5pt
\belowdisplayskip=7.5pt
\begin{align}
\sup\limits
_{n\geqslant\beta}\frac{1}{n}\sum\limits
_{i=1}^{n}
\mathbb {E} \bigl({\rm e}^{\gamma\xi_{i}}\ind_{\{\xi_i>0\}} \bigr)&\leqslant\sup
\limits
_{n\geqslant\beta}\frac{1}{n}\sum\limits
_{i=1}^{n}
\mathbb {E} {\rm e}^{\gamma Z_{i}}\leqslant\nu_1,\label{cc2}
\\
\max\limits
_{1\leqslant n\leqslant\beta-1}\frac{1}{n}\sum\limits
_{i=1}^{n}
\mathbb{E} \bigl({\rm e}^{\gamma\xi_{i}}\ind_{\{\xi_i>0\}} \bigr)&\leqslant
\xch{\nu_2,}{\nu_2} \label{cc2+}
\end{align}
because of the nonnegativity of the inter-occurrence times $\theta_i,
i\in\mathbb{N}$.

For the use of Theorem~\ref{th1}, it remains to estimate
\[
\sup\limits
_{n\geqslant\beta}\frac{1}{n}\sum\limits
_{i=1}^{n}
\mathbb {E} \xch{\bigl(|\xi_{i}|\ind_{\{\xi_i\leqslant-c\}} \bigr),}{(|\xi_{i}|\ind_{\{\xi_i\leqslant-c\}} )}
\]
for some suitable positive $c$.

Choosing $c=\varkappa$ we get
%
\begin{align}
\label{cccc} \sup\limits
_{n\geqslant\beta}\frac{1}{n}\sum\limits
_{i=1}^{n} \mathbb {E} \bigl(|\xi_{i}|
\ind_{\{\xi_i\leqslant-\varkappa\}} \bigr)&\leqslant p \sup\limits
_{n\geqslant\beta} \frac{1}{n}\sum
\limits
_{i=1}^{n}\mathbb {E} \bigl( \theta_i
\ind_{\{Z_i-p\,\theta_i\leqslant-\varkappa\}} \bigr)
\nonumber
\\
& =p \sup\limits
_{n\geqslant1}\frac{1}{n}\sum\limits
_{i=1}^{n}
\mathbb {E} \bigl(\theta_i\ind_{\{\theta_i\geqslant\frac{1}{p}(Z_i +\varkappa)\}
} \bigr)
\nonumber
\\
&\leqslant p \sup\limits
_{n\geqslant1}\frac{1}{n}\sum\limits
_{i=1}^{n}\mathbb{E} \bigl(\theta_i
\ind_{\{\theta_i\geqslant\frac
{\varkappa}{p}\}} \bigr)
\nonumber
\\
& \leqslant p\, \epsilon.
\end{align}
\endgroup

The obtained inequalities \eqref{cc1}, \eqref{cc2}, \eqref{cc2+} and
\eqref{cccc} imply that r.v.'s $\{\xi_1,\xi_2,\ldots\}$ satisfy
conditions (i)--(iv) of Theorem~\ref{th1} with
\[
a=\alpha,\quad b=\beta,\quad c=\varkappa,\quad h=\gamma,\quad d_1=1+\nu_1,\quad d_2=1+\nu_2\quad \mbox{and}\quad \varepsilon=p\,\epsilon.
\]
The assertion of the corollary follows now from Theorem~\ref{th1}.
\end{proof}

\section{Examples}\label{ex}

In this section, we present four examples. The first two examples show
the applicability of Theorem~\ref{th1}. The third example demonstrates
how to get the exponential bound for the ruin probability applying
Corollary~\ref{ccc2}. The last example shows that the Lundberg-type
inequality of the form $\psi(u)\leqslant\varrho_1{\rm e}^{-\varrho_2 u}$, $u\geqslant
0$, with $\varrho_1=1$ and  a positive constant $\varrho_2$ is impossible\querymark{Q2} if the
inhomogeneous renewal risk
model is ``good'' only on average.

\begin{ex}\label{ex11} Suppose that $\{\xi_1,\xi_2,\ldots\}$
are independent r.v.'s such that:
\begin{itemize}\setlength\itemsep{0pt}
\item[$\bullet$] $\xi_i$ are uniformly distributed on interval $[0,2]$
for all $i\,\equiv\,1\, \mbox{mod}\,3$;
\item[$\bullet$] $\xi_i$ are uniformly distributed on interval
$[-2,0]$ for all $i\,\equiv\,2\, \mbox{mod}\,3$;
\item[$\bullet$] $\overline{F}_{\xi_i}(x)=\ind_{(-\infty,-2)}(x)+\mbox
{e}^{-x-2}\ind_{[-2,\infty)}(x)$ if $i\,\equiv\,0\, \mbox{mod}\,3$.
\end{itemize}
\end{ex}

We can see that the presented sequence $\{\xi_1,\xi_2,\ldots\}$ has
three subsequences. Two of them generate random walks with negative
drifts, and one subsequence generates random walk with a positive
drift. Fortunately, sequence $\{\xi_1,\xi_2,\ldots\}$ has a negative
drift on average. Therefore, we can use Theorem~\ref{th1} to get the
exponential bound for $\mathbb{P}(M_\infty>x)$.

It is evident that
\begin{eqnarray*}
\mathbb{E}\xi_i= %
\begin{cases}
1& \text{if}\ i\,\equiv\,1\, \mbox{mod}\,3,\\
-1& \text{if}\ i\,\equiv\,2\, \mbox{mod}\,3\ \mbox{or}\ i\,\equiv\,0\,\mbox{mod}\,3.
\end{cases} %
\end{eqnarray*}
Therefore, after some calculations, we get
%
\begin{equation}
\label{*} \frac{1}{n}\sum\limits
_{i=1}^n
\mathbb{E}\xi_i\leqslant-\frac{1}7 \ \mbox {for}\ n
\geqslant7.
\end{equation}
Additionally,
%
\begin{equation}
\label{**} \sup\limits
_{n\geqslant1}\frac{1}{n}\sum\limits
_{i=1}^{n} \mathbb{E} \bigl(|\xi_i|
\ind_{\{\xi_i\leqslant-2\}} \bigr)=0
\end{equation}
and
\begin{eqnarray*}
\mathbb{E} \bigl(\mbox{e}^{\frac{4}{5}\xi_i}\ind_{\{\xi_i>0\}} \bigr)=
\begin{cases}
\frac{5}{8} (\mbox{e}^{8/5}-1 )<2.48 &\mbox{if}\ \  i\,\equiv\, 1\, \mbox{mod}\,3,\\
0 &\mbox{if}\ \  i\,\equiv\,2\, \mbox{mod}\,3,\\
5/\mbox{e}^2<0.68 &\mbox{if}\ \  i\,\equiv\,0\, \mbox{mod}\,3.
\end{cases} %
\end{eqnarray*}
The last expression implies
%
\begin{align}
&\sup\limits
_{n\geqslant7}\frac{1}{n}\sum\limits
_{i=1}^{n}
\bigl(\mathbb{P}(\xi_i\leqslant0)+\mathbb{E} \bigl(
\mbox{e}^{\frac{4}{5}\xi
_i}\ind_{\{\xi_i>0\}} \bigr) \bigr) <1.79,\label{***}
\\
& \max\limits
_{1\leqslant n\leqslant6}\frac{1}{n}\sum\limits
_{i=1}^{n}
\bigl( \mathbb {P}(\xi_i\leqslant0)+\mathbb{E} \bigl(
\mbox{e}^{\frac{4}{5}\xi_i} \ind_{\{
\xi_i>0\}} \bigr) \bigr) <2.48.\label{****}
\end{align}

By \eqref{*}--\eqref{****}, we conclude that conditions of Theorem~\ref
{th1} hold with
$a=1/7$, $b=7$, $c=2$, $\varepsilon=0$, $h=4/5$, $d_1=1.8$ and
$d_2=2.5$. Therefore,
\[
-\varDelta=\varepsilon+\delta h d_1\max \biggl\{\frac{c^2}{2},
\frac
{2}{h^2} \biggr\}-a=\frac{9}{2}\,\delta-\frac{1}7=-
\xch{\frac{1}{14},}{\frac{1}{14}}
\]
if $\delta=1/63$. It follows now from Theorem~\ref{th1} that
\begin{align*}
\mathbb{P}(M_\infty>x)&\leqslant\min \biggl\{1,\, \biggl(d_2
\frac{d_2^{\,
6}-1}{d_2-1}+\frac{\exp\{-\delta{h}\varDelta b\}}{1-\exp\{-\delta{h}\varDelta
\}} \biggr)\,{\rm e}^{-\delta{h} x} \biggr\}
\\
&\leqslant\min \xch{\bigl\{1,\, 1502 \exp\{-0.01269 x\} \bigr\},}{\bigl\{1,\, 1502 \exp\{-0.01269 x\} \bigr\}}
\end{align*}
for all positive $x$.

The last inequality works if $x\geqslant578$. Though the obtained
bound is not as good as one would prefer, it is still exponential. Its
weakest point is the large constant before the main term. By Theorem~\ref{th1}, the value of this constant is closely related to the
behavior of the first elements of the sequence $\{\xi_1,\xi_2,\ldots\}
$. In our case, the first elements of $\{\xi_1,\xi_2,\ldots\}$ increase
this constant because the subsequence $\{\xi_1, \xi_4,\xi_7,\ldots\}$
has a positive drift. The second example shows that the better
exponential bound can be obtained from Theorem~\ref{th1} in the case
when only some of the r.v.'s $\{\xi_1,\xi_2,\ldots\}$ drag the model to
the positive side.

\begin{ex}\label{ex12}Suppose that $\{\xi_1,\xi_2,\ldots\}$ are
independent r.v.'s such that:
\[
\mathbb{P}(\xi_i=-1)=1-\frac{1}{i+1}\quad \text{and}\quad \mathbb{P}(\xi_i=1)=\frac{1}{i+1}\quad \text{for}\ i\in\xch{\{1,2,\ldots\}.}{\{1,2,\ldots\}}
\]
\end{ex}

For all $i\geqslant1$, we have
\[
\mathbb{E}\xi_i=\frac{2}{i+1}-1\quad \text{and}\quad \mathbb{E} \bigl({\rm e}^{\xi_i}\ind_{\{\xi_i>0\}} \bigr)=
\frac{{\rm e}}{i+1}.
\]
Consequently,
\begin{align*}
\frac{1}{n}\sum\limits_{i=1}^n\mathbb{E}\xi_i&\leqslant-\frac{5}{18}\quad \text{if}\ n\geqslant3,
\\
\sup\limits
_{n\geqslant3}\frac{1}{n}\sum\limits
_{i=1}^{n}
\bigl(\mathbb{P}(\xi_i\leqslant0)+\mathbb{E} \bigl(
\mbox{e}^{\xi_i}\ind_{\{\xi
_i>0\}} \bigr) \bigr) &=\frac{13{\rm e}+23}{36}<1.621,
\\
\max\limits
_{1\leqslant n\leqslant6}\frac{1}{n}\sum\limits
_{i=1}^{n}
\bigl(\mathbb{P}(\xi_i\leqslant0)+\mathbb{E} \bigl(\mbox
{e}^{\xi_i}\ind_{\{\xi_i>0\}} \bigr) \bigr)&=\frac{{\rm e}+1}{2}<1.86.
\end{align*}

Due to the derived bounds conditions of Theorem~\ref{th1} hold with
\[
a=5/18,\quad\!\! b=3,\quad\!\! c=11/10,\quad\!\! \varepsilon=0,\quad\!\! h=1,\quad\!\! d_1=1.625\quad\!\! \text{and}\quad\!\! d_2=\frac{{\rm e}+1}{2}.
\]
In this case, we have
\[
-\varDelta=\varepsilon+\delta h d_1\max \biggl\{\frac{c^2}{2},
\frac
{2}{h^2} \biggr\}-a=\frac{13}{4}\,\delta-\frac{5}{18}<\xch{0,}{0}
\]
for $\delta<10/117$.

If we chose $\delta=1/20$, then by Theorem~\ref{th1}, we have
\begin{eqnarray*}
\mathbb{P}(M_\infty>x)\leqslant\min \bigl\{1,\, 178 \exp\{- x/20\} \bigr
\},\quad x\geqslant0.
\end{eqnarray*}

As was stated before, the next example shows the possibility of the
exponential bound for the ruin probability in the case when the
inhomogeneous renewal risk model satisfies net profit condition on average.

\begin{ex}\label{ex13} Let us consider the inhomogeneous risk
model which is generated by uniformly distributed on $[1,3]$
inter-occurrence times $\theta_1,\theta_2,\dots$,
constant premium rate $p=2$ and a sequence of the claim amounts $\{
Z_1,Z_2,\ldots\}$ such that $Z_1=Z_2=0$, $Z_3=Z_4=4$ with probability 1
and
\begin{eqnarray*}
&&\overline{F}_{Z_i}(x)=\ind_{(-\infty, 0)}(x)+{\rm e}^{-x}
\biggl(1+\frac
{x}{i}\, \biggr)\ind_{[0,\infty)}(x),\quad i\geqslant5.
\end{eqnarray*}
\end{ex}

In this case, we have
\begin{align*}
\mathbb{E}Z_i&=1+\frac{1}{i},\quad i \geqslant5,\\
\mathbb{E}\,{\rm e}^{\gamma Z_i}&=\frac{i-1}{i}\frac{1}{1-\gamma
}+
\frac{1}{i}\frac{1}{(1-\gamma)^2},\quad i\geqslant5,\ \gamma\in(0,1).
\end{align*}
Consequently,
\begin{eqnarray*}
&& \frac{1}{n}\sum\limits
_{i=1}^{n}(\hspace{0.5mm}
\mathbb{E}Z_i-p\mathbb{E}\theta _i)\leqslant-{2}\quad \text{if}\ \xch{n\geqslant1}{n\geqslant1,}
\end{eqnarray*}
and
\[
\sup\limits
_{n\geqslant1}\frac{1}{n}\sum\limits
_{i=1}^{n}
\mathbb{E} \bigl({\rm e}^{Z_i/3} \bigr)\leqslant2.4.
\]

The obtained inequalities imply conditions of Corollary~\ref{ccc2} with
\[
\alpha=2,\quad \beta=1,\quad \varkappa=6,\quad \epsilon=0,\quad \gamma=\frac{1}3,\quad \nu_1=2.4,\quad \nu_2=1.
\]
Therefore, for the described model,
\[
-\widehat{\varDelta}=\frac{102}{5}\, \delta-2<\xch{0,}{0}
\]
if $\delta<10/102$.

If we choose $\delta=9/102$, then $\widehat{\varDelta}=1/5$, and the
assertion of Corollary~\ref{ccc2} implies the following Lundberg-type
inequality for the model
\[
\psi(u)\leqslant\min \bigl\{1,\, 170\,{\rm e}^{- 0.029\,u} \bigr\},\quad u\geqslant0.
\]

If we choose $\delta=5/102$, then $\widehat{\varDelta}=1$, and Corollary~\ref{ccc2} implies that
\[
\psi(u)\leqslant\min \bigl\{1,\, 61\,{\rm e}^{- 0.016\,u} \bigr\},\quad u\geqslant0.
\]

\begin{remark}
It is clear that we can get a lot of
different Lundberg-type inequalities for the same model because there
exist infinitely many collections of constants $\{ \alpha, \beta,
\varkappa, \epsilon, \gamma, \nu_1, \nu_2, \delta\}$ satisfying
conditions of Corollary~\ref{ccc2}. It follows from the construction of
the bound in Corollary~\ref{ccc2} that we get the better bound for the
smaller constants $\beta, \varkappa, \epsilon, \nu_1, \nu_2 $ and for
larger constants $\alpha, \gamma$. If the collection of the constants
$\{ \alpha, \beta, \varkappa, \epsilon, \gamma, \nu_1, \nu_2\}$ is
quite ``unfriendly'', then we can still get an exponential bound for
the ruin probability but with unsatisfiably small $\delta$.
All possible exponential bounds for the ruin probability have the form
$\varrho_1\exp\{-\varrho_2 u\}$ with some positive constants $\varrho
_1$ and $\varrho_2$. Theorem~\ref{tt2} shows that $\varrho_1=1$ in the
case of the homogeneous renewal risk model satisfying the net profit
condition. If the net profit condition holds on average (see condition
(i) of Corollaries~\ref{ccc1} or~\ref{ccc2}), then it is impossible to
get the exponential bound for $\psi(u)$ with $\varrho_1=1$ in general.
The following simple example confirms this.
\end{remark}

\begin{ex}\label{ex14}
Let us consider the inhomogeneous risk
model with $p=1$ such that $Z_1=Z_2=10$, $Z_i=0$ for $i\geqslant3$,
and $\theta_i=1$ for $i\geqslant1$ almost surely.
\end{ex}

The model under consideration is inhomogeneous but satisfies the net
profit condition on average and all other conditions of, for instance,
Corollary~\ref{ccc2}. On the other hand, the model is degenerate. It is
not difficult to obtain the exact values of the ruin probability.
Namely, expression \eqref{pp} implies that
\begin{align*}
\psi(u)&=1\quad \text{if}\ 0\leqslant u< 18,\\
\psi(u)&=0\quad \text{if}\ u \geqslant18.
\end{align*}
We can see the graph of the function $\psi$ in Figure~\ref{p4}. All the
best possible exponential bounds for $\psi(u)$ must go through the
point A(18,1) (see colored curves in Figure~\ref{p4}). Therefore, the
upper bounds of the form $\varrho_1\exp\{-\varrho_2 u\}$ should satisfy
the condition
\[
\varrho_1{\rm e}^{-18 \,\varrho_2}\geqslant1.
\]

Hence, it is evident that
\[
\varrho_1\geqslant{\rm e}^{18\,\varrho_2 }>1.
\]

\begin{figure}
\includegraphics{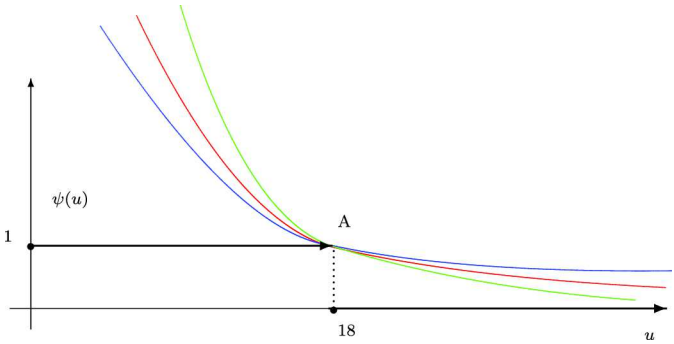}
\caption{Ruin probability in the model of Example~\ref{ex14}}\label{p4}
\end{figure}

\section{Concluding remarks}\label{cr}

In the paper, the problem of the estimating of the ruin probability for
the inhomogeneous renewal risk models is considered. It is evident that
this problem is closely related to the bounds for the tail probability
of the supremum of an inhomogeneous random walk. The upper bound of the
exponential type $\varrho_1{\rm e}^{-\varrho_2u}$ is derived for the
renewal risk models satisfying the \textit{net profit condition on
average}. The positive constants $\varrho_1$ and $\varrho_2$ depend on
the constants describing the model. Unfortunately, the obtained
estimates are not sharp enough. We guess that it is possible to get
sharper exponential bounds for the ruin probabilities but for narrower
class of the inhomogeneous renewal risk models.

\section*{Acknowledgments} We would like to thank three anonymous
referees for the detailed and helpful comments on the previous versions
of the manuscript.


\end{document}